\documentclass{amsart}
\usepackage{hyperref, color}

\theoremstyle{plain}
\newtheorem{theorem}                {Theorem}      [section]
\newtheorem*{theorem*}                {Theorem \ref{thm:appl}}
\newtheorem{proposition}  [theorem]  {Proposition}
\newtheorem{corollary}    [theorem]  {Corollary}

\begin{document}

\title[The $r$-mean curvature and rigidity of compact hypersurfaces]
{The $r$-mean curvature and rigidity of compact hypersurfaces in the Euclidean space}

\author[Costa-Filho]{Wagner Oliveira Costa-Filho}
\address{Campus Arapiraca, Federal University of Alagoas, CEP 57309-005, 
Arapiraca,  Alagoas, Brazil}
\email{fcow@bol.com.br}

\subjclass[2010]{Primary 53C42; Secondary 53C24, 53C44.}
\date{\today}
\keywords{isometric immersions; higher order mean curvature; self-shrinkers; 
$\lambda$-hypersurfaces; Minkowski integral formulas.}

\begin{abstract}  
In this paper, we characterize round spheres in the Euclidean space under some suitable conditions
on the  $r$-mean curvature.
\end{abstract}

\maketitle

\section{Introduction}

Let  $x:M^n\to\mathbb{R}^{n+1}$ be an isometric immersion of a orientable Riemannian manifold 
$M^n$ in the Euclidean space   $\mathbb{R}^{n+1}$. 
Denote by $A$ the second fundamental form of the hypersurface with respect to a unit normal vector
field $N$ globally defined on $M$.
The  $r$-mean curvature of  $M$ is defined by 
$$H_r=\binom{n}{r}^{-1}S_r,$$
where $S_r$ is the  $r$-elementary symmetric function of the eigenvalues of $A$, for 
 $r=1,\ldots ,n$, and  $S_0=1$.

The Newton transformations  $P_r$ related to the immersion $x$ 
are the linear maps defined recursively by  $P_0=I$ and $P_r=S_rI-AP_{r-1}$,
when $1\le r \le n$. 
Associated to each $P_r$  we have the linear differential operators of second order 
$L_r$ given by
$$L_r(u)=\textrm{tr}(P_r\nabla^2 u).$$
In the right hand side of this equation,  $\nabla^2u$ stands for the Hessian operator
of a smooth function  $u$  on $M$. 
It is well known that the operators $L_r$ are elliptic if and only if 
the corresponding  Newton transformations  $P_r$ 
are positively defined. 
Moreover, 
$$L_r(u)=\textrm{div}(P_r(\nabla u)), $$ 
where  $\nabla u$ is the gradient of   $u$ and $\textrm{div}$ denotes the divergent operator of a vector
field on  $M$. 
Classically, the operators $P_r$ and $L_r$ come from the variational aspects related to the
problem of minimizing certain $r$-area functionals  of the immersion $x$. 
For more details see seminal paper of Rosenberg  \cite{rosenberg1993}.

A still challenging topic in differential geometry is the rigidity of of hypersurfaces $M^n$
in the Euclidean space under some natural conditions on the topology and under 
suitable analytical assumptions on  $H_r$ or $L_r$, for some  $r=0,...,n$. 
In this context it is expected  to prove that the immersion is in fact totally umbilical. 
We point out that there exists a vast literature on this subject, exploring the cases where
$x$ is an embedding or just an immersion. For instance, the celebrated work of Aleksandrov \cite{Aleksandrov}
claims that a closed (compact and with empty boundary) embedded  hypersurface 
with constant mean curvature is a round sphere. This result was generalized by Ros \cite{Ros87}
in the case that some higher order mean curvature is constant. 
In the case that $x$ is an immersion and $M$ is a topological $2$-sphere in 
$\mathbb R^3$ with constant mean curvature,
 the also celebrated work of Hopf \cite{H} shows that $M$ a round sphere. It is well known that 
 for any integer $g\geq 1$ there are constant mean curvature surfaces with genus $g$ in
 $\mathbb R^3$ (see \cite{W} for $g=1$ and \cite{Ka2, Ka3} for $g\geq 2$).
 In higher dimensions,  Hsiang, Teng and Yu \cite{Hsiang} 
 proved the existence of topological spheres in $\mathbb R^{2n}$ 
 with constant mean curvature that are not round.

The aim of this work is to present new characterizations of the Euclidean sphere 
in terms of the behavior of the $r$-mean curvature $H_r$ when $x$ is
an  immersion. 

To state our first results, we denote by $\rho$ the support function of $x$, that is 
$\rho: M \to \mathbb{R}$,  $\rho = \langle x, N \rangle$.
Geometrically, $\rho(p)$ is the distance with sign from the origin $0\in\mathbb{R}^{n + 1}$ to the hyperplane tangent to $x(M)$ at $x(p)$.
Assuming that $\rho$ is non-negative, Deshmukh  \cite {deshmukh2013note} proved that the
mean curvature $H$ of $M$  is a solution to the Poisson equation $\Delta u = 1 + H_1 \rho$ 
if and only if, $M $ is isometric to a round sphere.
Our first theorem extends this result for the operators $L_r$.

\begin{theorem}
Let $x:M^n \to \mathbb {R}^{n + 1}$ be a closed hypersurface with non-negative support function. 
and such that operator $P_r$ is positively definite. 
Then the mean curvature $H_1$ satisfies the equation $ L_ru = 1 + H_1 \rho$ 
if and only if $M$ is isometric to a round sphere.
\end{theorem}

Assuming that $H_r$ is constant we obtain that

\begin{theorem}\label{t2}
Let $x: M^n \to \mathbb {R}^{n + 1}$ be a closed hypersurface with  non-negative support function.
Assume that,  for some $1 \le r \le{n-2}$, the operator $P_r$ is positively defined and  $H_r$ is constant.
Then the mean curvature $H_{r + 1}$ satisfies the equation $L_ru = H_r + H_{r + 1} \rho$ 
if and only if $M$ is isometric to a round sphere.
\end{theorem}

The positivity of the operator $P_r$ is a natural analytical condition, which is automatically verified 
when $r=0$. It is an interesting problem to prove Theorem \ref{t2} when $H_r$ is not constant.

\medskip
The techniques used to prove the theorems above can be applied to self-shrinkers of
the Euclidean space.  We recall  that $M^n\subset \mathbb R^{n+1}$ 
is a \textit{self-shrinker} if the equation is satisfied
\begin{equation}
    H = - \frac{\rho}{2},
\end{equation}
where $H$ is the \emph{non-normalized} mean curvature of $M$.

Self-shrinkers form an important class of solutions for the mean curvature flow 
and stand out in the study of the so-called type I singularities. 
See, for example, Colding and Minicozzi \cite{colding2012generic}.

Some basics examples of self-shrinkers are hyperplanes passing through the origin, minimal cones, round spheres $\mathbb{S}^n(\sqrt{2n})$ and cylinders $\mathbb{S}^k(\sqrt{2k})\times \mathbb{R}^{n-k}$, for $k = 1, ..., n-1$.

In \cite{guo2018scalar}, Guo obtained some gap theorems for closed self-shrinkers and concluded
that  if  the  scalar curvature of such hypersurfaces is constant, then they are isometric to a round sphere. 
In the following, we present a direct and more general result.

\begin{theorem}\label{t3}
Let $x: M^n \to \mathbb{R}^{n + 1}$ be a closed self-shrinker with $H_{r + 1}$ 
constant for some $1\le r\le n-1$. Then $M = \mathbb{S}^n(\sqrt{2n})$.
\end{theorem}

As a natural extension  of self-shrinkers, we say that $M$ is a \textit{$\lambda$-hipersurface} 
if the following equation is satisfied 
\begin{equation}
    H = - \frac{\rho}{2} + \lambda,
\end{equation}
where $\lambda$ is a constant.
For example, the sphere $\mathbb{S}^n(r)$  with radius $r$ is a $\lambda$-hypersurface   in $\mathbb{R}^{n + 1}$ for  $ \lambda = n/r-r/2$.

This concept  was introduced by Cheng and Wei in \cite{cheng2018complete}
where they studied mean  curvature flow that preserves a weighted volume. 
The authors show, among other facts, that a compact $\lambda$-hypersurface  is isometric to a round sphere if $H-\lambda \ge0$ and $\lambda(f(H-\lambda)-|A|^2) \ge0$, 
where $|A|^2 = \sum_{i,j}h_{ij}^2$ is the square of the norm of the second fundamental form 
and $f = \sum_{i, j, k}h_{ij}h_{jk} h_{ki}$.

Applying the same approach as in the proof of Theorem \ref{t3}, 
we obtain a simple proof of the following theorem due to Ross  \cite{ross2015rigidity}.

\begin{theorem}
Let $ x: M^n \to \mathbb {R}^{n + 1}$ be a $\lambda$-hypersurface closed with $H \ge \lambda$. If $|A|^2 \le 1/2$, then $M$ is a round sphere.
\end{theorem}


\section{Preliminaries}

In order to obtain our rigidity results, we  need the following propositions, which besides being important in themselves, have several other applications in  problems involving higher order mean curvatures
of hypersurfaces. We emphasize that such propositions are valid in space forms.

\begin{proposition}
Let $ x: M^n \to \mathbb{R}^{n + 1}$ be an orientable hypersurface of the
Euclidean space. If $ \rho: M\to  \mathbb{R}$ is the support function of $x$, then
\begin{equation}\label{lr}
L_r(\rho) = -(r + 1)S_{r + 1} -(S_1S_{r + 1} - (r + 2)S_{r + 2})\rho - \langle \nabla S_{r + 1}, x^T \rangle,
\end{equation}
where $x^\top$ indicates the component of $x$ tangent to $M$.
\end{proposition}
\begin{proof}
See Alencar e Colares \cite{alencar1998integral}, page 209.
\end{proof}

\begin{corollary}
Let $ x: M^n \to \mathbb{R}^{n + 1}$ be an orientable hypersurface of the
Euclidean space. If $ \rho: M \to \mathbb{R}$ is the support function of $x$, then
\begin{equation}\label{lapla}
    \Delta \rho = -nH_1-|A|^2\rho-n \langle \nabla H_1, x^T \rangle.
\end{equation}
\end{corollary}
\begin{proof}
Take $r = 0$ in equation (\ref{lr}) and use the identity $|A|^2 + 2S_2 = S_1^2.$

\end{proof}

The so-called Garding and Newton inequalities are used to prove the following
result:

\begin{proposition}\label{Garding}
Let $ x: M^n \to \mathbb{R}^{n + 1}$ be a closed orientable hypersurface. If $ H_{r + 1}$  is positive on $ M $, then for every $ i $, with $ 1 \le i \le r $, we have:
\begin{enumerate}
    \item[(a)] Each $ H_i $ is positive.
    \item[(b)] $ H_1H_{i + 1} -H_{i + 2} \ge0 $.
\end{enumerate}
Moreover, equality in $ (b) $ occurs for some $ i $ if, and only if
$M$ is a round sphere.
\end{proposition}
\begin{proof}
See Silva et al. \cite{da2016stability}, page 297.
\end{proof}

In the next result we present the classical Minkowski integral formula. 
For the sake of completeness, we present a concise demonstration following ideas of Alias and Malacarne in \cite{alias2004first}.

\begin{proposition}\label{Minkowski}
Let $ x: M^n \to \mathbb{R}^{n + 1} $ be a closed hypersurface. 
Then for every $ r $, with $ 0 \le r \le n-1 $, we have $$ \int_M (H_r + H_{r + 1} \rho) dM = 0.$$
\end{proposition}
\begin{proof}
Consider the function $ g: M^n \to \mathbb{R} $ defined by $ g = (1/2)|x|^2. $ We know that $ \nabla g = x^T $, where $ x^T = x- \rho N $. Then, for each tangent vector field $ X $ on $ M $ we have 
$$ (\nabla^2g)(X) = \nabla_X (\nabla g) = X + \rho A (X). $$
Therefore, $$ L_r (g) = tr (P_r \nabla^2g) = tr (P_r) + tr (AP_r) \rho = c_rH_r + c_rH_{r + 1} \rho = c_r (H_r + H_{r + 1 } \rho), $$
with $ c_r = (n-r) \binom {n}{r} $ and the traces are determined in \cite{rosenberg1993}, page 13. By the divergence theorem, it follows that $$ \int_M (H_r + H_{r + 1 } \rho) dM = 0, $$
finalizing the proof.
\end{proof}

To conclude this section we present two identities that will be useful
for our purposes. First, a directly computation yields 
\begin{equation}\label{laplaciano}
    \Delta |x|^2 = 2n (1 + H_1 \rho).
\end{equation}

The next identity is a consequence of the divergence theorem.
\begin{equation}\label{divergencia}
    \int_MuL_r (v) dM = - \int_M \langle P_r (\nabla u), \nabla v \rangle dM,
\end{equation}
whenever  $ u $ and $ v $ are smooth functions on  $ M $

\section{Proof of Theorems}

In this section we present the proofs of our  theorems. 
For the reader's convenience, we state the theorems again.

\begin{theorem}
Let $x:M^n \to \mathbb {R}^{n + 1}$ be a closed hypersurface with non-negative support function. 
and such that operator $P_r$ is positively definite. 
Then the mean curvature $H_1$ satisfies the equation $ L_ru = 1 + H_1 \rho$ 
if and only if $M$ is isometric to a round sphere.
\end{theorem}
\begin{proof}
If $ H_1 $ is a solution to the equation we have $ H_1L_rH_1 = H_1 + H_1^2 \rho$. So, applying identity (\ref{divergencia}) we get
$$ 
- n \int_M \langle P_r (\nabla H_1), \nabla H_1 \rangle dM = n \int_MH_1dM + n \int_MH_1^2 \rho dM. 
$$
On the other hand, using formula (\ref{lapla}) 
$$ n \int_MH_1dM = - \int_M |A|^2 \rho dM + \frac {n}{2} \int_M H_1 \Delta |x|^2dM. $$
From (\ref{laplaciano}) and the hypothesis about $ H_1, $ we rewrite this last equality as 
$$ 
n \int_MH_1dM = - \int_M |A|^2 \rho dM-n^2 \int_M \langle P_r (\nabla H_1), \nabla H_1 \rangle dM. 
$$
Therefore, 
$$
(n^2-n) \int_M \langle P_r (\nabla H_1), \nabla H_1 \rangle dM + \int_M (|A|^2-nH_1^2) \rho dM = 0. 
$$
Since $ P_r $ is  positively definite and $ \rho \ge 0 $, it follows that $ H_1 $ and $ \rho $ are constant. Furthermore, we conclude that $ M $ is totally umbilical, and therefore a round sphere.
\end{proof}

We now recall the following algebraic inequality related to $ r $-th mean curvature. For each  $ 1 \le r \le n- 1 $ it holds 
$$ 
H_r^2 \ge H_{r-1}H_{r + 1},
$$  and equality occurs only at umbilical points of $ M $. See, for example, Steele \cite{steele2004cauchy}, page 178.

\begin{theorem}
Let $x: M^n \to \mathbb {R}^{n + 1}$ be a closed hypersurface with  non-negative support function.
Assume that,  for some $1 \le r \le{n-2}$, the operator $P_r$ is positively defined and  $H_r$ is constant.
Then the mean curvature $H_{r + 1}$ satisfies the equation $L_ru = H_r + H_{r + 1} \rho$ 
if and only if $M$ is isometric to a round sphere.
\end{theorem}
\begin{proof}
As before, if $ H_{r + 1} $ is a solution to that equation, we get  
$$
- \int_M \langle P_r (\nabla H_{r + 1}), \nabla H_{r + 1} \rangle dM = \int_MH_rH_{r + 1} dM + \int_MH_{r + 1}^2 \rho dM. 
$$
Since $ H_{r + 1}^2 \rho \ge H_rH_{r + 2} \rho$, and using our 
hypotheses on  $ P_r $, $ H_r $ and the Minkowski formula we obtain
$$ 
0 \ge- \int_M \langle P_r (\nabla H_{r + 1}), \nabla H_{r + 1} \rangle dM \ge H_r\int_M (H_{r + 1} + H_{r + 2} \rho) dM = 0, 
$$
It follows that $ H_{r + 1} $ is constant and so all 
inequalities  above are equalities. It means that 
$H_r^2 = H_{r-1}H_{r + 1}$ on $M$, and 
we conclude that $ M $ is a round sphere.
\end{proof}

Now we prove our theorems on self-shrinkers.

\begin{theorem}
Let $x: M^n \to \mathbb{R}^{n + 1}$ be a closed self-shrinker with $H_{r + 1}$ 
constant for some $1\le r\le n-1$. Then $M = \mathbb{S}^n(\sqrt{2n})$.
\end{theorem}
\begin{proof}
Since $ H_{r + 1} $ is constant, we obtain by integrating the identity (\ref{lr})
\begin{eqnarray*}
 0 &=& - (r + 1) \binom{n}{r + 1} \int_M H_{r + 1} dM \\ 
 &&- \int_M \left[n \binom{n}{r + 1} H_1H_{r + 1} - (r + 2) \binom {n}{r + 2} H_{r + 2} \right] \rho \, dM \\ 
 &=& - (r + 1) \binom {n}{r + 1} \int_MH_{r+1} dM \\ 
 &&- \int_M \left[n \binom {n}{r + 1} H_1H_ {r + 1} - (n-(r + 1)) \binom {n}{r + 1} H_ {r + 2} \right] \rho \, dM
\end{eqnarray*}

Organizing the terms, $$ - (r + 1) \int_M (H_ {r + 1} + H_ {r + 2} \rho) dM-n \int_M (H_1H_ {r + 1} -H_ {r + 2} ) \rho \, dM = 0. $$

Therefore, from Proposition \ref{Minkowski} and by the equation of a self-shrinker we have, $$ \int_M (H_1H_ {r + 1} -H_ {r + 2})H dM = 0. $$

Choosing the orientation such that $ H_ {r + 1}> 0 $, we conclude by 
Proposition \ref{Garding} that $ M $ is totally umbilical. Therefore, $ M = \mathbb {S}^n (\sqrt {2n}) $.
\end{proof}

Finally, we will show that

\begin{theorem}
Let $ x: M^n \to \mathbb {R}^{n + 1}$ be a closed $\lambda$-hypersurface  with $H \ge \lambda$. If $|A|^2 \le 1/2$, then $M$ is a round sphere.
\end{theorem}
\begin{proof}
Fist, let us consider the case $\lambda \leq 0$. 
Since  $\rho = 2(\lambda - H)\leq 0$ we can use identity (\ref{lapla}) to obtain
\[
\Delta \rho -\frac 1 2  \langle \nabla \rho, x^\top\rangle - \frac{\rho}{2}
\geq -\lambda \geq 0.
\]
It follows from the strong maximum principle that $\rho$ is constant and thus
$H$ is also constant.  Now we can use  Minkowski formula and
identity (\ref{lapla}) to conclude that $n|A|^2 = H^2$, and so $M$
is totally umbilical.

\medskip

Now, let us assume that $\lambda \geq 0$. Since
$ H = nH_1 $ and $ \rho = 2(\lambda-H) $, we use  identity (\ref{lapla}) 
to get  
\begin{eqnarray*}
    0 &=& \int_MHdM + \int_M2 (\lambda-H) |A|^2dM + 
    \frac{1}{2} \int_M \langle \nabla H, \nabla |x|^2 \rangle dM \\ 
    &=& \int_MHdM + \int_M2(\lambda-H) |A|^2dM- \frac {1}{2} \int_MH \Delta |x|^2 dM \\ 
    &=& \int_MHdM + \int_M2(\lambda-H) |A|^2dM- \int_MH(n + 2 (\lambda-H)H)dM,
\end{eqnarray*}
where we have used formula (\ref{laplaciano} in the last equality.
Organizing the terms we obtain
\begin{equation}\label{last}
    \int_M [(n-1) H + 2 (H-\lambda)(|A|^2-H^2)]dM=0.
\end{equation}

Now, using  that $H^2\leq n|A|^2\leq 2n$ we conclude that
\begin{eqnarray*}
    (n-1) H + 2 (H-\lambda)(|A|^2-H^2) &\ge& 
    (n-1)H-2(H-\lambda) (n-1)|A|^2 \\
    &\ge& (n-1)H-(n-1)(H-\lambda) \\
    &=& (n-1) \lambda \ge0,
\end{eqnarray*}
In view of identity (\ref{last}) we conclude that all inequalities above
are in fact identities. In particular, $M$ is a round sphere.
\end{proof}

\noindent
{\bf Acknowledgments: }
The author would like to express his gratitude to 
Prof. Marcos P. Cavalcante for suggestions and 
much encouragement.

\bibliographystyle{amsplain}
\bibliography{bib}

\end{document}